\documentclass{article}

\usepackage{arxiv}

\usepackage[utf8]{inputenc} 
\usepackage[T1]{fontenc}    
\usepackage{hyperref}       
\usepackage{url}            
\usepackage{booktabs}       
\usepackage{amsfonts}       
\usepackage{nicefrac}       
\usepackage{microtype}      
\usepackage{graphicx}
\usepackage[english]{babel}
\usepackage[square,numbers]{natbib}
\usepackage{doi}
\usepackage{amssymb}
\usepackage{amsmath}
\usepackage{color}
\usepackage{csquotes}
\usepackage{amsthm}
\usepackage{bm}
\usepackage{cleveref}
\usepackage{subfigure}

\newtheorem{theorem}{Theorem}[section]

\newtheorem{definition}[theorem]{Definition}

\title{Mathematical Modeling of the Role of Imitation in Crime Dynamics}

\author{ \href{https://orcid.org/0000-0002-0508-152X}{\includegraphics[scale=0.06]{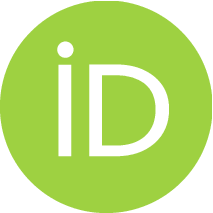}\hspace{1mm}Zeray H.~Gebrezabher} \\
	Faculty of Engineering and Natural Science\\
	Kadir Has University\\
	Istanbul, Turkey \\
	\texttt{zeray.hagos@khas.edu.tr} \\
	\And
	\href{https://orcid.org/0000-0001-6725-6949}{\includegraphics[scale=0.06]{orcid.eps}\hspace{1mm}Deniz Eroglu} \\
	Faculty of Engineering and Natural Science\\
	Kadir Has University\\
	Istanbul, Turkey \\
	\texttt{deniz.eroglu@khas.edu.tr} \\
}

\date{}

\hypersetup{
pdftitle={Mathematical Modeling of the Role of Imitation in Crime Dynamics},
pdfsubject={math},
pdfauthor={Zeray H.~Gebrezabher, Deniz Eroglu},
pdfkeywords={Crime model, Imitation, Basic reproduction number, Equilibria, Sensitivity analysis},
}

\begin{document}
\maketitle

\begin{abstract}
	Crime remains one of the significant problems that countries are grappling with globally. With shrinking economies and increasing poverty, crime has been on the rise in many countries. In this paper, we propose a system of non-linear ordinary differential equations to model crime dynamics in the presence of imitation. The model consists of four independent compartments: individuals who are not at risk of committing a crime, individuals at risk of committing a crime, individuals committing a crime, and individuals convicted and jailed for a crime. The model is analyzed using the basic reproduction number. The analysis shows the system has a locally asymptotically stable crime-free equilibrium when the basic reproduction number is less than unity. The model exhibits a backward bifurcation in which two endemic equilibria coexist with the crime-free equilibrium. When the basic reproduction number exceeds unity, the system has a locally asymptotically stable endemic equilibrium, and the crime-free becomes unstable. Numerical simulations are carried out to verify the analytical results. The sensitivity analysis shows that the relapse rate highly influences the basic reproduction number of our model. This indicates that the proportion of individuals leaving prisons and becoming criminals should be minimized to minimize crime.
\end{abstract}



\section{Introduction}
\label{sec:intro}

Crime remains a significant problem globally. With shrinking economies and increasing poverty, crime has been on the rise in many countries. It is considered a contagious disease that spreads from one individual to another due to direct or indirect contact quickly or gradually depending on a host of factors in the world \cite{Goldstick2022,doi:10.2105/AJPH.2016.303550,doi:10.1177/00223433211026126,Forum2013}. According to work in \cite{Mallony}, \enquote{many of the leaders of organized crime groups begin their criminal careers by committing petty crimes as they try to imitate a criminal role model}. Over time, they commit more serious crimes, become gang members, and eventually work their way up the ladder to become trusted members of a gang organization \cite{Mallony}. Moreover, it has been demonstrated in \cite{Bandura} that the imitation effect exists in children who copy on-screen violence following exposure.

Mathematical modeling is fundamental to understanding, controlling, and predicting the long-term behaviors of various aspects of dynamic systems in real-world scenarios. Examples include modeling the spread of crime or violence, social media addiction, and corruption, to mention just a few \cite{mca24010029,doi:10.1007/s11135-017-0581-9,Mamo2021,Danford2020,doi:10.1155/2022/8073877,Shari-boomChap2016}. It has been extensively formulated and analyzed the dynamics of crime or violence in terms of infectious disease dynamics models using a set of differential equations \cite{Misra:2014, Goyaletal2014, Gilbertoetal2018, Sooknanan,mca24010029,doi:10.1007/s11135-017-0581-9,Shari-boomChap2016}. For example, the authors in \cite{Sooknanan} proposed an infectious disease model to analyze the growth of gangs in a population. In their model, the population was divided into four groups based on gang status and risk factors concerning gang membership. Similarly, the authors in \cite{Gilbertoetal2018} developed a mathematical model of crime that considers a social endemic transmission. Their model has divided the population into six classes: susceptibles, free criminals, criminals arrested and in jail, convicted criminals, judges, and police officers. However, none of the previous mathematical models of crime dynamics incorporates the effect of imitation. Therefore, we aim to fill this gap by studying aspects of peer influence as the driver of criminal recruitment. 

In this work, we use aspects of peer influence as the driver of the recruitment of criminals. Researchers have modeled peer influence on alcohol consumption in \cite{Buonomo2}. We propose a crime model in the presence of imitation described by a system of ordinary differential equations. Using the proposed model, we aim to depict how individuals within a community might adopt criminal activities by observing and imitating their peers. We seek to replicate real-world scenarios in which the decision to engage in criminal acts can be influenced by observing others committing similar deeds. Through mathematical simulations, researchers can explore the impact of imitation on crime dynamics, shed light on potential intervention strategies, identify potential hotspots of criminal activity, and ultimately contribute to developing more effective crime prevention policies.

The paper is organized as follows. In section \enquote{Mathematical model formulation}, model assumptions are presented, and a corresponding mathematical model of the crime dynamics in the presence of imitation is formulated. A rigorous mathematical analysis of the proposed crime model is analyzed and presented in terms of the \enquote{basic reproduction number} in section \enquote{Mathematical analysis of the crime model}. The equations will then be analyzed for the stability of the steady states, and the existence of backward bifurcation and sensitivity analysis of the model parameters will also be discussed. Section \enquote{Numerics} is devoted to numerical simulations to corroborate the theoretical results. The concluding remarks of the paper are provided in section \enquote{Conclusion}.

\section{Mathematical model formulation}
\label{sec:model-formulation}

We proposed a mathematical model of ordinary differential equations to investigate how imitation of a particular criminal act affects the spread of crime between individuals. To start with, we define the state variables of the model. In formulating our model, the total population, denoted by $N(t)$ at any time $t$, is divided into four distinct classes or state variables:
\begin{itemize}
    \item $S_1(t)$: the number of individuals who are not at risk of committing a crime (Law-abiding),
    \item $S_2(t)$: the number of individuals who are at risk of committing a crime (Susceptible),
    \item $C(t)$: the number of individuals who are committing a crime (criminals or offenders), and
    \item $R(t)$: the number of individuals convicted and jailed for the crime (Incarcerated).
\end{itemize}
In such a way that we have the relation
\begin{align*}
    N(t) = S_1(t) + S_2(t) + C(t) + R(t).
\end{align*}

\subsection{Model assumptions}
To control the movement of individuals between the classes, we consider the following assumptions on the model parameters:
\begin{itemize}
    \item The population is recruited at a rate of $\pi$.
    \item Individuals enter the state $S_2(t)$ at a rate of $p\pi$ and the state $S_1(t)$ at a rate of $(1-p)\pi$, where $p\in (0,1)$ is the proportion of individuals entering the susceptible population.
    \item Individuals may move from state $S_2(t)$ to state $S_1(t)$ without committing a crime at a rate of $\varepsilon$ and back to state $S_2(t)$ at a rate of $\theta$. Therefore, the number of individuals at risk of committing a crime may increase due to the population transferring from the state $S_1(t)$.
    \item The rate at which the individual becomes criminal is assumed to follow an \emph{imitation process}, described comprehensively in \cite{Buonomo2}. To this end, we assume that the rate at which individuals join the class $C(t)$ gives initiation function $f(S_2(t), C(t)) = \beta S_2(t) C(t)(1+\alpha C(t))$ that is driven by imitation, with $\beta$ as the effective contact rate and $\alpha$ as the imitation coefficient.
    \item We do not allow the movement of criminal individuals back to the susceptible group $S_2(t)$. Also, it is not realistic to move non-susceptible (or law-abiding) individuals to be incarcerated without committing a crime.
    \item Individuals may move from state $C(t)$ to state $R(t)$ at a rate of $\sigma$, where $\sigma$ is a conviction rate.
    \item Once an individual is convicted and jailed for a crime, an individual can relapse back to the criminal state individuals, $C(t)$, at a rate of $q\gamma$ or may join the state $S_1(t)$ at a rate $(1-q)\gamma$, where $\gamma$ is the relapse rate and $q\in (0,1)$ is the proportion of individuals moving out of the $R(t)$ state to state $C(t)$.
    \item The time spent in the community by individuals has a constant average duration of $1/\mu$, where $\mu$ is the natural mortality rate, which can be caused by illness, suicide, accidental self-injury, execution, or any other unspecified cause.
    \item We assumed that all members of the population mix homogeneously. This implies that each individual has an equal chance of becoming a criminal.
    \item Since we are dealing with a model of human populations, all the model parameters are assumed to be non-negative. A summary of the description of the proposed model parameters is provided in Table~\ref{tab:1}.
\end{itemize}

\begin{table}[ht]
    \centering
    \caption{Description of model parameters of the crime model \eqref{eq:1}}
    \label{tab:1}
    \begin{tabular}{cl}
    \toprule
        \textbf{Parameter} & \textbf{Description}  \\
        \hline
        $\pi$ & Rate of recruitment of individuals \\
        $\mu$ & Natural mortality rate from any class \\
        $\theta$ & Rate of transfer from $S_1(t)$ to $S_2(t)$  \\
        $\epsilon$ & Rate of transfer from $S_2(t)$ to $S_1(t)$  \\
        $\sigma$ & Conviction rate \\
        $\beta$ & Effective contact rate   \\
        $\alpha$ & Imitation coefficient  \\
        $\gamma$ & Release rate from $R(t)$ to $C(t)$ and/or $S_1(t)$ \\
        $q$ & Proportion of relapse from $R(t)$ to $C(t)$ \\
        $p$ & Proportion of recruitment to $S_2(t)$\\
        \bottomrule
    \end{tabular}
\end{table}{}
\begin{figure}[ht!]
\centering
\includegraphics[width=\linewidth]{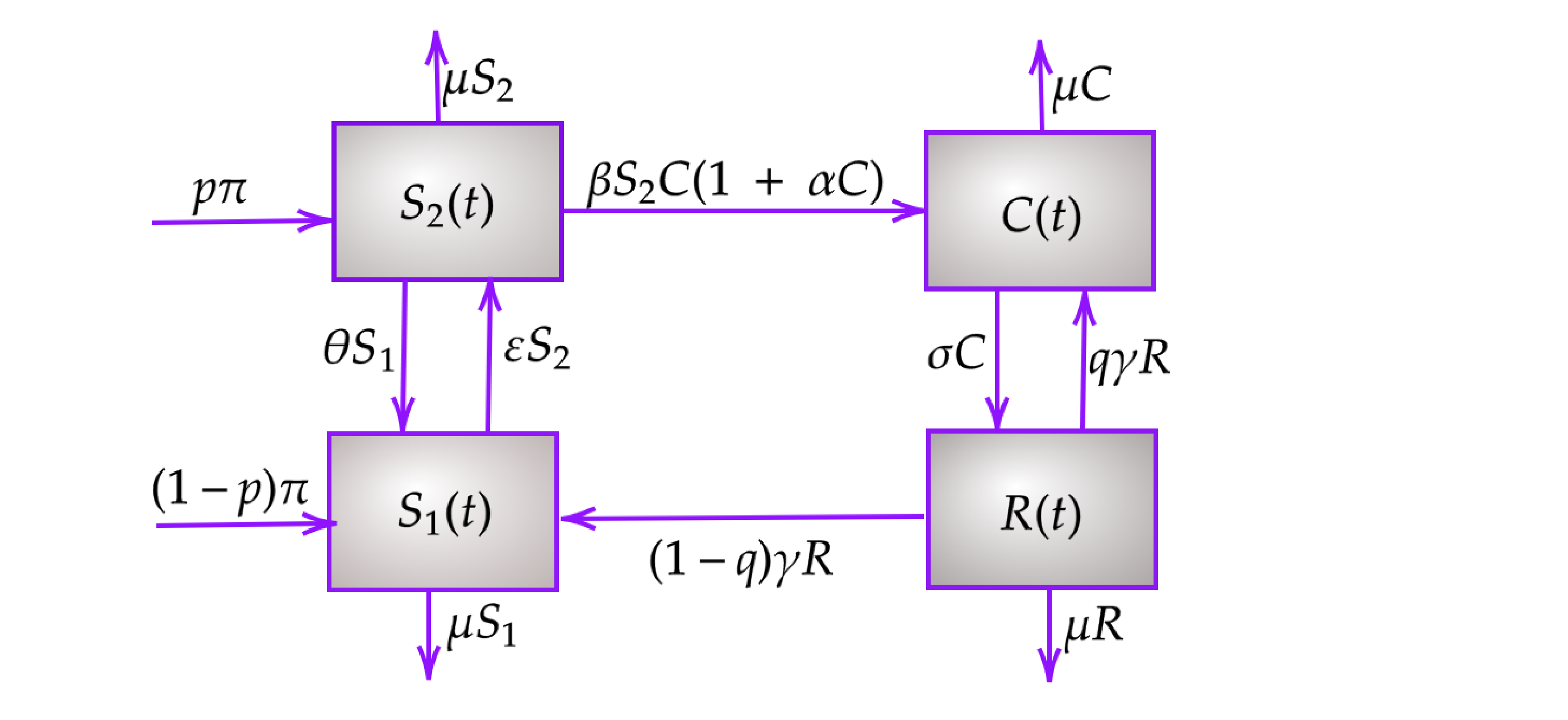}%
\caption{A flow diagram showing the schematic illustration of the model considered in this work.}%
\label{fig:1}%
\end{figure}
%
Based on the flow diagram in Fig.~\ref{fig:1} and the underlying model assumptions, it follows that the following system of nonlinear ordinary differential equations gives the proposed crime model in the presence of imitation: 
\begin{equation}\label{eq:1}
    \begin{split}
    \frac{d S_1}{d t} &= (1-p)\pi + (1-q)\gamma R -(\mu + \theta)S_1+\varepsilon S_2,\\
    \frac{d S_2}{d t} &= p\pi - \beta S_2 C(1+\alpha C)-(\mu +\varepsilon)S_2 +\theta S_1, \\
	\frac{d C}{d t} &= \beta S_2 C(1+\alpha C) +q\gamma R -(\mu +\sigma) C,\\
	\frac{d R}{d t} &= \sigma C - (\mu + \gamma)R,
    \end{split}
\end{equation}
with initial conditions $S_1(0)>0, S_2(0)\geq 0, C(0)\geq 0, R(0)\geq 0$.

\section{Mathematical analysis of the crime model}
\label{sec:analysis}

In this section, our model~(\ref{eq:1}) is qualitatively analyzed to investigate the existence of its equilibria [28] and the control strategies of its dynamical behavior.

We first observe from system~(\ref{eq:1}), by summing up the equations, that
\begin{eqnarray}
    \frac{d N(t)}{d t}= \pi - \mu N(t).\label{eq:N_t_dynamics}
\end{eqnarray} 
It follows that $\frac{d N(t)}{d t} < 0$ if $N(t) > \frac{\pi}{\mu}$. Solving Eq.~(\ref{eq:N_t_dynamics}) with an initial condition $N(0)$, we obtain
\begin{align*}
    N(t) = \left(N(0)-\frac{\pi}{\mu} \right)^{-\mu t} + \frac{\pi}{\mu}.
\end{align*}
It follows that
\[\limsup_{t\to \infty} N(t) \leq \pi/\mu.\]
Hence, the dynamics of system~(\ref{eq:1}) can be studied in the following feasible region:
\begin{align*}
\Omega = \big\{(S_1,S_2,C,R)\in \mathbb{R}_{+}^{4}:0\leq S_1+S_2+C+R \leq \frac{\pi}{\mu} \big\}
\end{align*}
which is \emph{positively invariant} with respect to the flow induced by system~\ref{eq:1}. Hence, our crime model~(\ref{eq:1}) is \enquote{well-posed} mathematically and epidemiologically \cite{Hethcote2000}. Hence, it is sufficient to study the dynamics of the crime model in $\Omega$.

\subsection{Existence of equilibrium solutions}
\label{subsec:equilibria}

In this subsection, we discuss the existence of \emph{equilibrium solutions} of system~(\ref{eq:1}), which are the time-constant solutions. A quadruple will represent a generic equilibrium
\begin{align}
    \Bar{E} = (\Bar{S}_1, \Bar{S}_2, \Bar{C}, \Bar{R}).
\end{align}
In analogy with the term \enquote{\emph{disease-free} equilibrium} of infectious disease models, we call a quadruple equilibrium solution of the form $E^0=(S_{1}^{0}, S_{2}^{0}, 0, 0)$ the \emph{crime-free} equilibrium of system~(\ref{eq:1}). In that case, the component of the criminality class $C(t)=0$ for all time $t$, and the entire population will comprise only susceptible individuals. Equilibrium solutions, or equilibria, $E^*=(S_1^*, S_2^*, C^*, R^*)$, with positive crime dynamics component ($C^* >0$), will be called \emph{endemic} equilibria.
\begin{theorem}[Existence of crime-free equilibrium]
The model~(\ref{eq:1}) has a crime-free equilibrium state, $E^0$, at the point
\begin{eqnarray*}
    E^0=\left(\frac{\pi((1-p)\mu + \varepsilon)}{\mu (\mu +\varepsilon +\theta)},\frac{\pi(p\mu +\theta)}{\mu (\mu +\varepsilon +\theta)} ,0,0\right).
\end{eqnarray*}
\end{theorem}
\begin{proof}
    At the equilibrium state, the right-hand side of the system \eqref{eq:1} is set to zero. That is, if $\Bar{E} =(\Bar{S}_{1},\Bar{S}_{2},\Bar{C},\Bar{R})$ is an equilibrium state of the model \eqref{eq:1}, then we must have
\begin{eqnarray}
0 &=& (1-p)\pi + (1-q)\gamma \Bar{R} -(\mu + \theta)\Bar{S}_{1} +\varepsilon \Bar{S}_{2}, \label{eq-3a}\\
	0 &=& p\pi - \beta \Bar{S}_{2} \Bar{C}(1+\alpha \Bar{C})-(\mu +\varepsilon)\Bar{S}_{2} +\theta \Bar{S}_{1}, \label{eq-3b}\\
	0 &=& \beta \Bar{S}_{2} \Bar{C}(1+\alpha \Bar{C}) +q\gamma \Bar{R} -(\mu +\sigma) \Bar{C}, \label{eq-3c}\\
	0 &=& \sigma \Bar{C} - (\mu + \gamma)\Bar{R}.\label{eq-3d}.
\end{eqnarray}
From \eqref{eq-3c} and \eqref{eq-3d}, one obtains
\[\Bar{C}=0 \quad \text{or}\quad \beta \Bar{S}_{2}(1+\alpha \Bar{C}) +\frac{q\gamma\sigma}{\mu + \gamma} -(\mu  +\sigma)=0.\]
But $\Bar{C}=0$ at a crime-free equilibrium state. Hence, after some manipulations, our model exhibits a crime-free equilibrium state at $E^0=\left(\frac{\pi((1-p)\mu + \varepsilon)}{\mu (\mu +\varepsilon +\theta)},\frac{\pi(p\mu +\theta)}{\mu (\mu +\varepsilon +\theta)},0,0\right)$.
\end{proof}
Motivated by the epidemiological models of infectious diseases (see, e.g., \cite{vanden}), our model is analyzed in terms of the basic reproduction number, $\mathcal{R}_0$--\enquote{the number of secondary infections caused by one infectious individual}, to show how individual involvement in a criminal act could be controlled to reduce the likelihood of an individual engaging in a criminal career. The basic reproductive number $\mathcal{R}_0$ of the model \eqref{eq:1} may be evaluated as the spectral radius of the so-called next-generation matrix \cite{vanden,Diekmann,5e095723cfcd4b7ba7088a8463fea99f}. Because we are concerned with the individuals that spread criminality activity, we can only consider the states $C$ and $R$ to obtain the \enquote{next-generation matrix} of our proposed crime model. To this end, we obtain
\begin{eqnarray}
    \mathcal{R}_0=\frac{\beta\pi(p\mu+\theta)(\mu + \gamma)}{\mu(\mu+\theta+\varepsilon)\Lambda},\label{eq:2}
\end{eqnarray}
where
\begin{eqnarray}\label{eq:3}
    \Lambda = (\mu+\gamma)(\mu+\sigma) - q\gamma \sigma.
\end{eqnarray}
Define the following quantities:
\begin{eqnarray}
    \alpha^{*} &:=& \frac{(\mu+\gamma)(\mu+\sigma)-\gamma\sigma(q\mu+\theta)}{\pi(\mu+\gamma)(p\mu+\theta)},\nonumber \\
    \mathcal{R}_{0}^c &:=& 1-\frac{\beta(\alpha - \alpha^*)^2}{4\alpha\Phi}, \label{eq:4}
\end{eqnarray}
where
\begin{equation}
\Phi=\left[\mu^2(\mu+\theta+\sigma+\gamma)+\mu\theta(\sigma+\gamma)+(1-q)\gamma\sigma\mu\right]\mu(\mu+\theta+\varepsilon)\Lambda. \label{eq:5}
\end{equation}
In the following theorem, based on the basic reproduction number $\mathcal{R}_0$, we investigate the existence of endemic equilibria of our model \eqref{eq:1}.

\begin{theorem}\label{thm:endemic}[Existence of endemic equilibrium]
The following statements hold.
\begin{itemize}
    \item[(i)] If $\mathcal{R}_0>1$, then model \eqref{eq:1} has a unique endemic equilibrium.
    \item[(ii)] If $\mathcal{R}_0<1$ and $\alpha <\alpha^{*}$, then model \eqref{eq:1} has no endemic equilibrium.
    \item[(iii)] If $\mathcal{R}_{0}^c<\mathcal{R}_0<1$ and $\alpha >\alpha^{*}$, then model \eqref{eq:1} has two endemic equilibra.
    \item[(iv)] If $\mathcal{R}_0<\mathcal{R}_{0}^c$ and $\alpha >\alpha^{*}$, then model \eqref{eq:1} has no endemic equilibrium.
\end{itemize}
\end{theorem}

\begin{proof}
At an endemic equilibrium state, crime always exists.
It can be shown that the components of the generic equilibrium of model \eqref{eq:1} are given by
\begin{align*}
    S_{1}^* &=\frac{\varepsilon \Lambda + D_s}{(\mu+\theta)(\mu+\gamma)\beta(1+\alpha C^{*})},\\
    S_{2}^*&=\frac{\Lambda}{(\mu + \gamma)\beta(1+\alpha C^{*})},\\
    R^{*} &= \frac{\sigma C^{*}}{\mu + \gamma},
\end{align*}
where we have used that
\[D_s = (1-p)(\mu+\gamma)\pi\beta(1+\alpha C^{*})+(1-q)\gamma\sigma\beta C^{*}(1+\alpha C^{*})\]
and $C^*$ is a positive solution of the quadratic equation
\begin{align}\label{eqx}
b_2(C^{*})^2+b_1C^{*}+b_0=0,
\end{align}
where
\begin{align*}
b_0&=\mu(\mu+\theta+\varepsilon)\Lambda(\mathcal{R}_0-1),\\
b_1&=\beta\pi(\mu+\gamma)(p\mu+\theta)(\alpha - \alpha^*),\\ 
b_2&=-\beta\alpha[\mu^2(\mu+\theta+\sigma+\gamma)+\mu\theta(\sigma+\gamma)+(1-q)\gamma\sigma\mu].
\end{align*}
It turns out that
\begin{align*}
    b_2 <0, \quad b_1 >0 \iff \alpha > \alpha^*, \quad b_0 >0 \iff \mathcal{R}_0 > 1.
\end{align*}
Moreover, the critical threshold $\mathcal{R}_0^c$ satisfies the discriminant $\Delta:= b_1^2 - 4b_0 b_2 = b_{1}^2+4\beta\alpha\Phi(\mathcal{R}_0-1)=0$. It follows that $\Delta >0 \iff \mathcal{R}_0>\mathcal{R}_{0}^c$. By the Descartes rule of signs, 
\begin{itemize}
\item if $\mathcal{R}_0>1$, i.e., $b_0>0$, then \eqref{eqx} has a single positive solution $C^*$,
\item if $\mathcal{R}_0<1$, i.e., $b_0<0$ and $b_1<0$, then \eqref{eqx} has no positive solution.
\item if $\mathcal{R}_0<\mathcal{R}_{0}^c$ and and $\alpha >\alpha^{*}$, then \eqref{eqx} has no positive solution,
\item if $\mathcal{R}_{0}^c<\mathcal{R}_0<1$ and $\alpha >\alpha^{*}$, i.e., $b_1>0$, then \eqref{eqx} has two positive solutions, say $C_1^*$ and $C_2^*$.
\end{itemize}
Hence, the statements in (i), (ii), (iii), and (iv) hold.

\end{proof}

\subsection{Local stability of the crime-free equilibrium}
\label{subsec:loc-stability}

In the following theorem, we show the local stability of the crime-free equilibrium of system \eqref{eq:1}.
\begin{theorem}\label{thm:cfe}[Local stability of the crime-free equilibrium]
The crime-free equilibrium $E^0$ of system \eqref{eq:1} is locally asymptotically stable if $\mathcal{R}_0< 1$, and it is unstable if $\mathcal{R}_0>1$.
\end{theorem}{}
\begin{proof}
The Jacobian matrix of system \eqref{eq:1} evaluated at the crime-free equilibrium $E^0$, denoted by $\bm{J}(E^0)=: \bm{J}$ is given by
\begin{align*}
\bm{J} &=\begin{pmatrix}
-(\mu+\theta) &\varepsilon &0 &(1-q)\gamma\\
\theta & -(\mu+\varepsilon) & -\frac{\mathcal{R}_0\Lambda}{\mu+\gamma} & 0\\
0 & 0 & \frac{(\mathcal{R}_0 - 1)\Lambda - q\gamma\sigma}{\mu+\gamma} & q\gamma\\
0 & 0 & \sigma & -(\mu+\gamma)
\end{pmatrix}\\
&= \begin{pmatrix}
    \bm{X} & \bm{Z} \\
    \bm{0} & \bm{Y}
\end{pmatrix}.
\end{align*}
Here $\bm{X},\bm Y, \bm Z$, and $\bm 0$ are $2\times 2$ block submatrices of the Jacobian $\bm J$, given by
\begin{eqnarray*}
    \bm{X} &=& \begin{pmatrix}
        -(\mu+\theta) &\varepsilon \\
        \theta & -(\mu+\varepsilon)
    \end{pmatrix},\; \bm{Y} = \begin{pmatrix}
        0 &(1-q)\gamma\\
        -\frac{\mathcal{R}_0\Lambda}{\mu+\gamma} & 0
    \end{pmatrix},\\ \bm Z &=& \begin{pmatrix}
        \frac{(\mathcal{R}_0 - 1)\Lambda - q\gamma\sigma}{\mu+\gamma} & q\gamma\\
        \sigma & -(\mu+\gamma)
    \end{pmatrix}, \; \bm 0 = \begin{pmatrix}
        0 & 0\\ 0 & 0
    \end{pmatrix}.
\end{eqnarray*}
Hence, the Jacobian $\bm J$ is a block upper triangular matrix. Therefore, the eigenvalues denoted $\lambda$, of $\bm J$ equal the eigenvalues of the submatrices $\bm X$ and $\bm Y$. In particular, the eigenvalues of $\bm J$ are solutions of the two quadratic equations:
\begin{align*}
    &\lambda^2 +(2\mu+\theta+\varepsilon)\lambda+\mu(\mu+\theta+\varepsilon)=0 ,\; \text{and} \\
    &\lambda^2 + a_1\lambda + a_2 =0,
\end{align*}
where
\begin{align*}
    a_1 &= \frac{(\mu+\gamma)^2 + q\gamma\sigma + (1-\mathcal{R}_0)\Lambda}{\mu+\gamma},\\
    a_2 &= (1-\mathcal{R}_0)\Lambda.
\end{align*}
Using the quadratic formula, the eigenvalues of the first equation are given by
\begin{align*}
\lambda_1, \lambda_2= \frac{1}{2}\left(-(2\mu+\theta+\varepsilon)\pm \sqrt{(\theta+\varepsilon)^2-2\theta\varepsilon} \right).
\end{align*}
Since $\theta+\varepsilon \geq \sqrt{(\theta+\varepsilon)^2-2\theta\varepsilon} = \sqrt{\theta^2 + \varepsilon^2}$, it follows that $\lambda_1<0$ and $\lambda_2 <0$. Similarly, the eigenvalues of the second equation are given by
\begin{align*}
\lambda_3, \lambda_4= \frac{1}{2}\left(-a_1\pm \sqrt{a_{1}^{2}-4a_2} \right).
\end{align*}
Notice that, $a_2 > 0 \iff \mathcal{R}_0 < 1$. So, if $a_2 >0$, then $\sqrt{a_1^2 - 4a_2} \leq \sqrt{a_1^2}=|a_1|$. On the other hand, if $a_2<0$, then $\sqrt{a_1^2 - 4a_2} \geq \sqrt{a_1^2}=|a_1|$. Hence, it can be easily seen that both eigenvalues $\lambda_3, \lambda_4$ are negative only if $a_1>0$ and $a_2>0$. Thus, all the eigenvalues $\lambda_i$ of $J(E^0)$ have negative real parts only if $\mathcal{R}_0 < 1$. By \cite{Hale}, the crime-free equilibrium $E^0$ is locally asymptotically stable if $\mathcal{R}_0 < 1$.
\end{proof}

\subsection{Bifurcation analysis}
\label{sec:bifurcation}

Theorem~\ref{thm:cfe} shows that $\mathcal{R}_0 = 1$ is a bifurcation value. The crime-free equilibrium changes its stability properties in a neighborhood of $\mathcal{R}_0 =1$. On the other hand, Theorem~\ref{thm:endemic} shows that system \eqref{eq:1} exhibits a backward bifurcation, that is, when $\alpha > \alpha^*$ and $\mathcal{R}_{0}^c<\mathcal{R}_0<1$ system \eqref{eq:1} exhibits two endemic equilibria (one locally stable and one unstable) together with a locally asymptotically stable crime-free equilibrium (see Fig.~\ref{fig:bfr}). This shows that it is not enough to lower the basic reproduction number below unity to eradicate criminality in the population.

As in \cite{Buonomo2}, the existence of backward bifurcation of system \eqref{eq:1} can also be verified using Theorem~4.1 of the approach in \cite{Castilo-Chavez:2004}, which is based on the \emph{center manifold theory} \cite{GuckeHolmes83}.

The sensitivity analysis, which we will illustrate in the following subsection, shows that $\beta$ is one of the most sensitive parameters. Let us choose $\beta$ as a bifurcation parameter. We see that
\begin{align}\label{eq:bifparameter}
    \mathcal{R}_0 = 1 \iff \beta=\beta^* := \frac{\mu(\mu+\theta+\varepsilon)\Lambda}{\pi(p\mu+\theta)(\mu+\gamma)}.
\end{align}
Linearization of system \eqref{eq:1} evaluated at the crime-free equilibrium, $E^0$, and plugging $\beta=\beta^*$, gives the Jacobian matrix
\begin{align*}
    \bm J^*(E^0, \beta^*) = \begin{pmatrix}
-(\mu+\theta) & \varepsilon  &  0  &  (1-q)\gamma\\
\theta  & -(\mu+\varepsilon) & \frac{ \Lambda}{\mu + \gamma}  &  0\\
0  &  0  &  -\frac{q\gamma \sigma}{\mu + \gamma} & q\gamma\\
0  &  0  &\sigma   &  -(\mu+\gamma)
\end{pmatrix}.
\end{align*}{}
The eigenvalues of $\bm J^*(E^0, \beta^*)$ are 
\begin{align*}
    &\lambda_1 = 0, \; \lambda_2 = -\frac{(\mu+\gamma)^2+q\gamma\sigma}{\mu+\gamma} <0,\\ 
    &\lambda_3,\lambda_4 = \frac{-(2\mu+\theta+\varepsilon)\pm \sqrt{(\theta+\varepsilon)^2-2\theta\varepsilon}}{2}.
\end{align*}
Since $\theta+\varepsilon \geq \sqrt{(\theta+\varepsilon)^2-2\theta\varepsilon}$, we see that the eigenvalues $\lambda_3 <0$ and $\lambda_4 < 0$ are negative. Moreover, $\lambda_1=0$ is a simple eigenvalue. Hence, when $\mathcal{R}_0=1$, the crime-free equilibrium $E^0$ is a non-hyperbolic equilibrium. Applying the center manifold theory \cite{GuckeHolmes83}, we analyse the dynamics of system \eqref{eq:1} near $\beta=\beta^*$ (i.e. $\mathcal{R}_0=1$). Let $\bm{v}=(v_1,v_2,v_3,v_4)$ and $\bm{w}=(w_1,w_2,w_3,w_4)^T$, respectively, denote the left and right eigenvectors associated with the zero eigenvalue $\lambda=0$, satisfying $\bm{v}\cdot\bm{w}=1$. After some algebraic manipulations, we can show that 
\begin{equation*}
    \bm{w} = \begin{pmatrix}
        \frac{(1-q)\gamma\sigma(\mu + \varepsilon) + \varepsilon\Lambda}{\sigma\mu(\mu + \theta + \varepsilon)}\\
        \frac{(1-q)\gamma\sigma\theta + \Lambda(\mu + \theta)}{\sigma\mu(\mu + \theta +\varepsilon)} \\
        \frac{\mu + \gamma}{\sigma} \\
        1
    \end{pmatrix} \; \text{and}\; \bm{v} = \begin{pmatrix}
         0\\ 0\\ \frac{\mu + \gamma}{q\gamma}\\ 1 
    \end{pmatrix}^T
\end{equation*}
Now, let $\bm{x}=(x_1,x_2,x_3,x_4)^T$ denotes the state vector of the system \eqref{eq:1} with $x_1 =S_1, x_2 =S_2, x_3 = C$, and $x_4 =R$, and $g_i$ denote the right-hand side of the dynamics of $x_i$, for each class $i=1,\dots,4$. Then the system \eqref{eq:1} comes into
\begin{align}\label{eq:dyn}
    \frac{d x_i}{d t} = g_i(\bm x),\quad i\in \{1,\dots,4\},
\end{align}
where
\begin{align*}
    g_1(\bm x) &= (1-p)\pi + (1-q)\gamma x_4 -(\mu + \theta)x_1+\varepsilon x_2,\\
    g_2(\bm x) &= p\pi - \beta x_2 x_3(1+\alpha x_3)-(\mu +\varepsilon)x_2 +\theta x_1,\\
    g_3(\bm x) &= \beta x_2 x_3(1+\alpha x_3) +q\gamma x_4 -(\mu +\sigma) x_3,\\
    g_4(\bm x) &= \sigma x_3 - (\mu + \gamma)x_4.
\end{align*}
According to \cite{Castilo-Chavez:2004,Buonomo2,Martina2018}, \enquote{the coefficients $a$ and $b$ of the normal form representing the system dynamics} \eqref{eq:dyn} on the center manifold are given by
\begin{align}\label{eq:a-and-b}
    \begin{split}
        a &= \sum_{k,i,j}^{4} v_k w_i w_j \frac{\partial^2 g_k}{\partial x_i \partial x_j}(E^0,\beta^*),\\
       b &= \sum_{k,i,j}^{4} v_k w_i \frac{\partial^2 g_k}{\partial x_i \partial \beta}(E^0,\beta^*).
    \end{split}
\end{align}
where each second partial derivatives $\frac{\partial^2 g_k}{\partial x_i \partial x_j}(E^0,\beta^*)$ and $\frac{\partial^2 g_k}{\partial x_i \partial \beta}(E^0,\beta^*)$ are evaluated at the crime-free equilibrium $E^0$ and at the point $\beta=\beta^*$. For the sake of simplicity, we can omit the arguments, $E^0, \beta^*$, in the partial derivatives, and after some algebraic calculations, we obtain that
\begin{align*}
    \frac{\partial^2 g_2}{\partial x_2 \partial x_3} &= -\beta^*, \; \frac{\partial^2 g_2}{\partial x_3^2} = -2\beta^*\alpha \frac{p\pi \mu +\theta\pi}{\mu (\mu +\varepsilon +\theta)}, \\
    \frac{\partial^2 g_3}{\partial x_2 \partial x_3} &= \beta^*, \; \frac{\partial^2 g_3}{\partial x_3^2} = 2\beta^*\alpha \frac{p\pi \mu +\theta\pi}{\mu (\mu +\varepsilon +\theta)}, \\
    \frac{\partial^2 g_2}{\partial x_3 \partial \beta^*} &= -\frac{p\pi \mu +\theta\pi}{\mu (\mu +\varepsilon +\theta)} , \; \frac{\partial^2 g_3}{\partial x_3 \partial \beta^*} = \frac{p\pi \mu +\theta\pi}{\mu (\mu +\varepsilon +\theta)},
\end{align*}
and all the remaining second partials are zero. Substituting these values into the expressions of $a$ and $b$ in eq.~\eqref{eq:a-and-b}, we obtain
\begin{align*}
    a &= 2\frac{(\mu(\mu + \sigma + \gamma)+(1-q)\gamma\sigma)(\mu + \gamma)\mu}{\pi(p\mu + \theta)((\mu +\gamma)^2 + q\gamma\sigma)}\bm{\Psi} > 0,
    \\
    b &= \frac{(\mu + \gamma)^2\pi(p\mu + \theta)}{((\mu + \gamma)^2 + q\gamma\sigma)\mu(\mu + \theta + \varepsilon)} > 0,
\end{align*}
where we have used that
\begin{equation*}
    \bm{\Psi}=(1-q)\gamma\sigma(\mu + 2\theta) + (\mu + \theta)\mu(\mu + \sigma + \gamma) + (\mu + \gamma)\mu\alpha\pi(p\mu + \theta).
\end{equation*}
All the required hypotheses of Theorem~4.1 \cite{Castilo-Chavez:2004} are satisfied. Hence, we established the following result.
\begin{theorem}\label{thm:bifurcation}
The crime model \eqref{eq:1} exhibits a backward bifurcation when $\mathcal{R}_0=1$.
\end{theorem}
We illustrate the result of Theorem~\ref{thm:bifurcation} numerically by creating a bifurcation curve around $\mathcal{R}_0=1$ as shown in Fig.~\ref{fig:bfr}, considering the estimated model parameter values: $\pi=13820,\beta=8.5\times10^{-6},\alpha=0.00018,\mu=0.01316,\varepsilon=0.88,\theta=0.01,\gamma=0.9,\sigma=0.6,p=0.2$, and $q=0.5$. For this case, we have that $\alpha^*=0.0035$.
\begin{figure}[ht!]
\centering
\includegraphics[width=\linewidth]{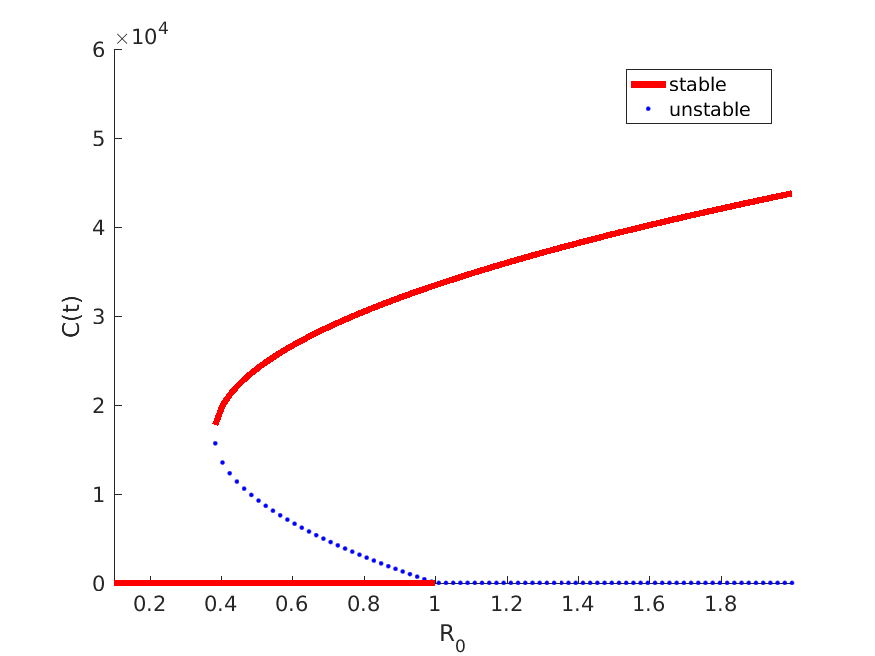}
\caption{\label{fig:bfr} Backward bifurcation curve in the plane $(\mathcal{R}_0,C)$. The estimated parameter values are, $\pi=13820,\beta=8.5\times 10^{-6},\alpha=0.0002,\mu=0.01316,\varepsilon=0.88,\theta=0.01,\gamma=0.9,\sigma=0.6,p=0.2$, and $q=0.5$. The threshold value is $\mathcal{R}_{0}^c = 0.3815$. It shows the occurrence of a backward bifurcation of system \eqref{eq:1}. This implies that when $\mathcal{R}_0<1$, a small positive unstable endemic equilibrium (dotted curve) appears while a crime-free and another positive endemic equilibrium (in red) are locally asymptotically stable. It can also be observed that the two endemic equilibria disappear when $\mathcal{R}_0$ is decreased below the critical value $\mathcal{R}_{0}^c<1$.}
\end{figure}
The implications of backward bifurcation are that $\mathcal{R}_0$ does not describe the necessary condition to clear crime when $\mathcal{R}_0<1$, but for crime to be eliminated from the population, we should lower the reproductive number below the critical threshold value, i.e., we require $\mathcal{R}_0<\mathcal{R}_{0}^c$. 

\subsection{Sensitivity analysis}
\label{subsec:sensitivity}

We use sensitivity analysis to investigate model parameters highly impacting $\mathcal{R}_0$. Following \cite{chitnis}, we compute the sensitivity indices of $\mathcal{R}_0$ to the parameters in the model \eqref{eq:1}. These indices show us how relevant each model parameter is in reducing criminality. 

The \enquote{normalized forward sensitivity index} (NFSI) \cite{chitnis} of $\mathcal{R}_0$ to a model parameter is the ratio of the relative change in the variable $\mathcal{R}_0$ to the relative change in the model parameter \cite{doi:10.3934/dcdsb.2018060}. When $\mathcal{R}_0$ is a differentiable parameter function, then the sensitivity index may be alternatively defined using partial derivatives.
\begin{definition}\cite{chitnis} \label{def-1}
The normalized forward sensitivity index of $\mathcal{R}_0$ that depends differentiablly on a parameter, $\chi$, is defined as:
\begin{align}\label{eq:NFSI}
\Psi_{\chi}^{\mathcal{R}_0}:=\frac{\chi}{\mathcal{R}_0}\frac{\partial\mathcal{R}_0}{\partial\chi}.
\end{align}
\end{definition}
Because increasing the mortality rate is neither ethical nor practical, $\mu$, we omit the sensitivity index of $\mathcal{R}_0$ to $\mu$.  Using equation~\eqref{eq:NFSI}, the sensitivity indices of $\mathcal{R}_0$ concerning each of the remaining model parameters are calculated as follows:
\begin{align*}
\Psi_{\beta}^{\mathcal{R}_0}&=\frac{\beta}{\mathcal{R}_0}\frac{\partial\mathcal{R}_0}{\partial\beta}=1,\\
\Psi_{\pi}^{\mathcal{R}_0} &= \frac{\pi}{\mathcal{R}_0}\frac{\partial\mathcal{R}_0}{\partial\pi} = 1,\\
\Psi_{\theta}^{\mathcal{R}_0}&=\frac{\theta}{\mathcal{R}_0}\frac{\partial\mathcal{R}_0}{\partial\theta}=\frac{\theta(1-p)\mu+\theta\varepsilon}{(\mu+\theta+\varepsilon)(p\mu+\theta)} >0,\\
\Psi_{\varepsilon}^{\mathcal{R}_0}&=\frac{\varepsilon}{\mathcal{R}_0}\frac{\partial\mathcal{R}_0}{\partial\varepsilon}=-\frac{\varepsilon}{\mu+\theta+\varepsilon} <0,\\
\Psi_{\gamma}^{\mathcal{R}_0}&=\frac{\gamma}{\mathcal{R}_0}\frac{\partial\mathcal{R}_0}{\partial\gamma}=\frac{\gamma(\mu+\sigma+q\mu\sigma)}{(\mu+\gamma)(\mu+\sigma-q\gamma\sigma)} >0,\\
\Psi_{\sigma}^{\mathcal{R}_0}&=\frac{\sigma}{\mathcal{R}_0}\frac{\partial\mathcal{R}_0}{\partial\sigma}=-\frac{\sigma(1-q\gamma)}{\mu+\sigma-q\gamma\sigma} <0,\\
\Psi_{p}^{\mathcal{R}_0}&=\frac{p}{\mathcal{R}_0}\frac{\partial\mathcal{R}_0}{\partial p}=\frac{p\mu}{p\mu+\theta}>0,\\
\Psi_{q}^{\mathcal{R}_0}&=\frac{q}{\mathcal{R}_0}\frac{\partial\mathcal{R}_0}{\partial q}=-\frac{q\gamma\sigma}{\mu+\sigma-q\gamma\sigma} <0.
\end{align*}
The sensitivity indices' positive (or negative) signs show direct (or indirect) proportionality. The model parameters $\pi$ and $\beta$ have a sensitivity index of $1$ for any choices of the remaining model parameters. This sensitivity index value shows that a $1\%$ increase in recruitment rate to the susceptible class and the effective contact rate will produce a $1\%$ increase of the basic reproduction number, $\mathcal{R}_0$. The parameters $\theta, \gamma$, and $p$ have positive sensitivity indices, whereas the parameters $\varepsilon, \sigma$, and $q$ have negative sensitivity indices. In Table \ref{tab:SIndex}, we show the sensitivity indices of the remaining model parameters using the estimated model parameter values provided in Table \ref{tab:model-params-b}. 

\begin{table}[ht!]
    \centering
    \caption{\label{tab:SIndex} Sensitivity indices of $\mathcal{R}_0$ corresponding to the model parameter values provided in Table \ref{tab:model-params-b}.}
    \begin{tabular}{c|c}
         \hline
         Model parameter &  Sensitivity index value \\
         \hline
         $\pi$ &   1\\
         $\beta$ &  1\\
         $\theta$ & 0.67451972\\
         $\epsilon$ &  -0.65934066\\
         $\gamma $ &  1.4243619\\
         $\sigma $ & -0.96835443\\
         $p$ &  0.02877698\\
         $q$ &  -0.4556962\\
         \hline 
    \end{tabular}
\end{table}
Table \ref{tab:SIndex} shows that the model parameter $\gamma$ has a sensitivity index of $1.4243619$, which means that a 1\% increase in the relapse rate will produce a 1.4243619\% increase in the basic reproduction number. On the other hand, the parameter $\sigma$ has a sensitivity index of $-0.96835443$, which shows that a $1\%$ increase in the conviction rate will produce a $-0.96835443\%$ decrease in the basic reproduction number. A similar interpretation can be made for the remaining model parameters.

\section{Numerics}
\label{sec:numerics}

\subsection{Estimated model parameters}
\label{subsec:model-params}

The estimation of parameters in any model validation process is a challenging task. We will consider hypothetical assumptions to illustrate the usefulness of our model in tracking the dynamics of crimes in the presence of imitation. We consider demographic parameters to estimate some of the model parameters. For the \emph{per capita} death rate $\mu$, we assume the average life expectancy of the human population is $75$ years. This value was approximately the life expectancy of Brazil in 2018 \cite{IBGE}. Hence, we consider $\mu=1/75$ per year. On average, the birth rate estimate of Brazil for the year 2018 was $13.82$ births per $1000$ people \cite{IBGE}. Using this data, we assume a birth rate of $13.82$ births per $1000$ people so that $\pi=0.015$ per year. For the parameters $p$ and $q$, we assume estimates on the interval $(0,1)$. The conviction rates usually fall from approximately $10\%$ to approximately $85\%$. Hence, we shall assume that $\sigma$ lies on the interval $(0.1,0.85)$. The relapse rate $\gamma$ is related to the duration of imprisonment. The maximum duration of imprisonment is $30$ years \cite{Pedro}. This duration varies from one prison to the other, so we assume that $\gamma$ lies on the interval $(0.03,1)$ for imprisonment ranging between a year and $30$ years. We shall assume that the contact rate $\beta$ to be taken from $(0,1)$. Finally, we shall estimate the rate of transfer of individuals between the compartments $S_1$ and $S_2$, $\theta$ and $\varepsilon$, to be lying on the intervals $(0.1, 0.8)$ and $(0,1)$, respectively. 

\begin{table}[!htb]
    \centering
    \caption{\label{tab:model-params-a} A summary of the estimated values of the model parameters for the case when $\mathcal{R}_0=0.8494 < 1$.}
    \begin{tabular}{c|c|c}
                \hline \hline
                \multicolumn{1}{c|}{Model parameter} & \multicolumn{1}{c}{Value} & \multicolumn{1}{|c}{Source} \\
                \hline
         $\pi$ &  $0.01382\times 10^6$ & \cite{IBGE} \\
         $\beta$ & $0.00000065$ & assumed\\
         $\alpha$ & $0.000002$ & assumed\\
         $\mu$ & $1/75$ & \cite{IBGE}\\
         $\epsilon$ & $0.2$ & assumed\\
         $\theta$ & $0.09$ & assumed\\
         $\gamma$ & $0.8$ & \cite{Pedro}\\
         $\sigma$ & $0.5$ & literature\\
         $p$ & $0.2$ & assumed\\
         $q$ & $0.4$ & assumed\\
                \hline
            \end{tabular}
\end{table}
\begin{table}[!htb]
    \centering
    \caption{\label{tab:model-params-b} A summary of the estimated values of the model parameters for the case when $\mathcal{R}_0=1.2883 > 1$.}
    \begin{tabular}{c|c|c}
        \hline \hline
\multicolumn{1}{c|}{Model parameter} & \multicolumn{1}{c}{Value} & \multicolumn{1}{|c}{Source} \\
                \hline
         $\pi$ &  $0.01382\times 10^6$ & \cite{IBGE}\\
         $\beta$ & $0.0000018$ & assumed\\
         $\alpha$ & $0.000002$ & assumed\\
         $\mu$ & $1/75$ & \cite{IBGE}\\
         $\epsilon$ & $0.2$ & assumed\\
         $\theta$ & $0.09$ & assumed\\
         $\gamma $ & $0.8$ & \cite{Pedro}\\
         $\sigma $ & $0.6$ & literature\\
         $p$ & $0.2$ & assumed\\
         $q$ & $0.4$ & assumed\\
                \hline
            \end{tabular} 
\end{table}

\subsection{Numerical simulations}
\label{subsec:simulations}

We hypothetically consider a total population size $N=10^6$ in all the experiments.

\paragraph{Numerical simulation for the local stability of crime-free equilibrium.} We simulate our crime model \eqref{eq:1} numerically for time $500 \ \text{years}$ with $0.01$ time-step starting from randomly and uniformly chosen initial conditions $(S_1(0), S_2(0), C(0), R(0))$ in the interval $[0, 10^6]$. We use an explicit Runge-Kutta order 5(4) method for numerical integration \cite{DORMAND198019}. In our simulations, we consider the model parameters in Table \ref{tab:model-params-a}. These parameter values correspond to $\mathcal{R}_0=0.6462 < 1$. Hence, we obtain multivariate time-series data $\{S_1(t), S_2(t), C(t), R(t)\}$ for a time $500$. The results are shown in Fig.~\ref{fig:crime-free}.
\begin{figure}[ht!]
\centering
\includegraphics[width=\linewidth]{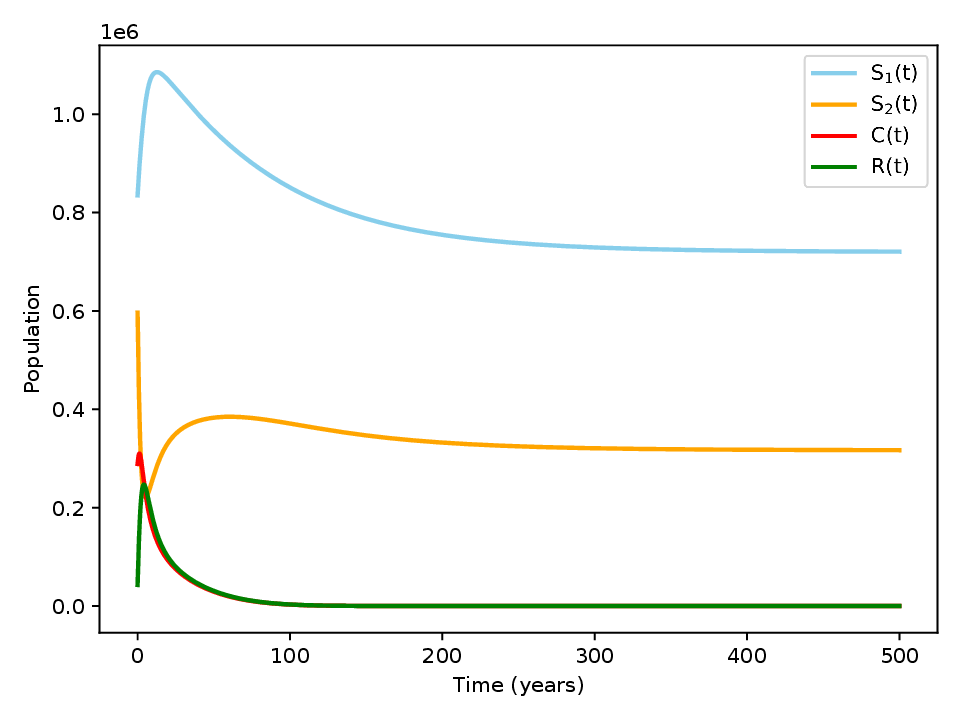}
\caption{\label{fig:crime-free} Time series of the state variables $S_1(t), S_2(t), C(t)$, and $R(t)$ showing local stability of the crime-free equilibrium $E^0$. The system \eqref{eq:1} was simulated numerically for time $500 \ \text{months}$ with $0.01$ time-step starting from randomly and uniformly chosen initial conditions $(S_1(0), S_2(0), C(0), R(0))$ in the interval $[0, 10^6]$. We considered the parameter values in Table \ref{tab:model-params-a}. These parameters correspond to $\mathcal{R}_0= 0.6462$. }
\end{figure}
Fig.~\ref{fig:crime-free} shows that the individuals who are not at risk of committing a crime, $S_1(t)$, and the individuals who are at risk of committing a crime, $S_2(t)$, initially decrease and blow up, that is, as time increases, the population also increases. On the other hand, the individuals committing crime, $C(t)$, and those who become convicted and jailed due to crime, $R(t)$, decline to zero. Therefore, the system approaches the crime-free equilibrium $E^0=(S_1^0, S_2^0, 0, 0)$. This implies that the crime-free equilibrium $E^0$ of model \eqref{eq:1} is locally asymptotically stable, which agrees with the analytical results. 

\paragraph{Numerical simulation showing the existence of endemic equilibrium.} Here, we perform numerical simulations of model \eqref{eq:1} when $\mathcal{R}_0 >1$ to show that crime persists in the population and that a unique endemic equilibrium exists. Here, the parameter values used are displayed in Table \ref{tab:model-params-b}. These values correspond to $\mathcal{R}_0=1.5108$. The results are illustrated in Fig. \ref{fig:endemic}.
\begin{figure}[ht!]
\centering
\includegraphics[width=\linewidth]{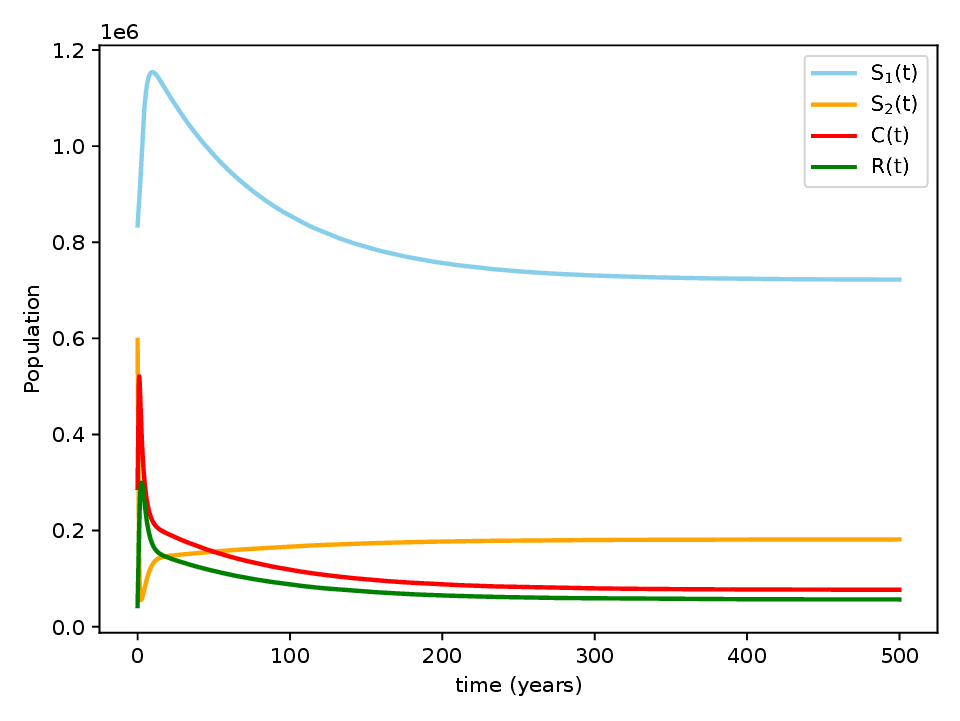}
\caption{\label{fig:endemic} Time series of the state variables $S_1(t), S_2(t), C(t)$, and $R(t)$ showing local stability of the crime-free equilibrium $E^*$. The system \eqref{eq:1} was simulated numerically for time $500 \ \text{months}$ with $0.01$ time-step starting from randomly and uniformly chosen initial conditions $(S_1(0), S_2(0), C(0), R(0))$ in the interval $[0, 10^6]$. We considered the parameter values in Table \ref{tab:model-params-b}. These parameters correspond to $\mathcal{R}_0=1.5108$. }
\end{figure}
Fig.~\ref{fig:endemic} shows that when $\mathcal{R}_0>1$, crime persists in the population, and the dynamics tend to an endemic equilibrium $E^*$. This shows that a unique endemic equilibrium $E^*$ is locally asymptotically stable, and the crime-free equilibrium $E^0$ becomes unstable when $\mathcal{R}_0>1$.

We further investigate the impact of the model parameters $\sigma$ (the conviction rate) and $\gamma$ (the release rate) on the basic reproduction number, $\mathcal{R}_0$, in reducing crime in the population. To see this, we use a contour plot (Fig.~\ref{fig:contour}) of $\mathcal{R}_0$ with varying the parameters $\sigma \in (0.10,0.85)$ and $\gamma \in (0.03, 1)$.

\begin{figure}[ht!]
\centering
\includegraphics[width=\linewidth]{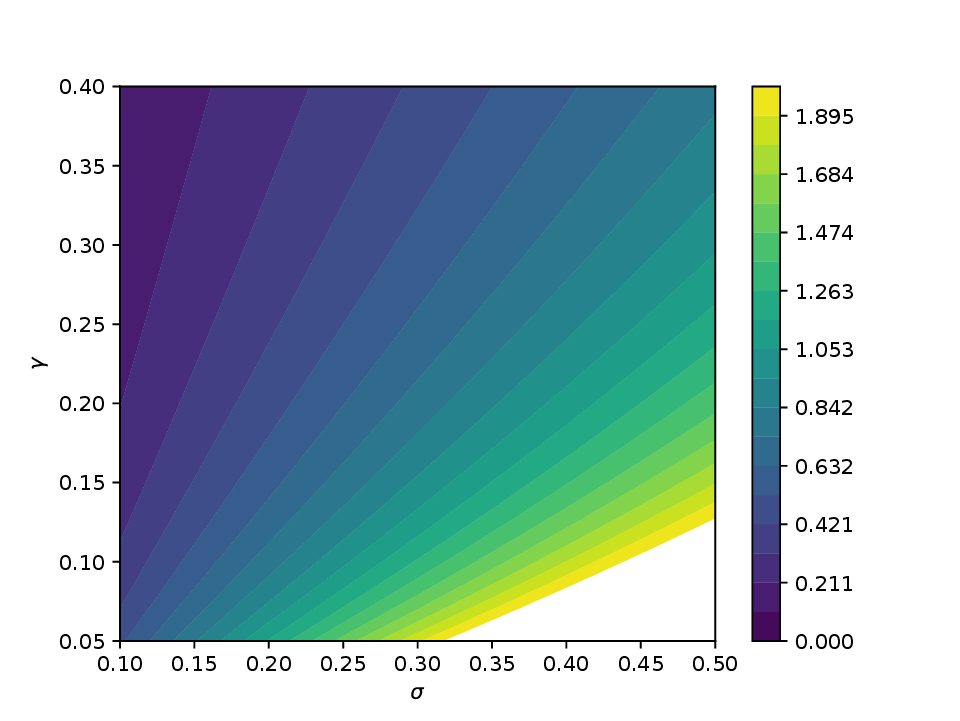}
\caption{\label{fig:contour} A contour plot how model parameters $\sigma$ and $\gamma$ affect $\mathcal{R}_0$. The parameters $\sigma$ and $\gamma$ are varied in the intervals $[0.1, 0.5]$ and $[0.05, 0.4]$, and the remaining parameters are as provided in Tables \ref{tab:model-params-a} and \ref{tab:model-params-b}.}
\end{figure}
Fig.~\ref{fig:contour} shows that $\mathcal{R}_0$ is highly dependent on the parameter $\gamma$, which is also verified using the sensitivity index of $\mathcal{R}_0$ to $\gamma$. This result shows that increasing $\gamma$ and decreasing $\sigma$ lead to a decrease in $\mathcal{R}_0$.

\paragraph{Effects of imitation in crime dynamics.} Our theoretical analysis shows we should lower the basic reproduction number $\mathcal{R}_0$ below the critical threshold value $\mathcal{R}_0^c$. We also notice that the quantity $\mathcal{R}_0$ is independent of the model parameter $\alpha$. Therefore, in this subsection, we will numerically investigate the effect of the imitation coefficient,$\alpha$, on the class $C(t)$ population dynamics. To this end, we vary the imitation coefficient in $\alpha\in \{10^{-5}, 0.0001, 0.0002\}$. We consider the remaining model parameters in Tables \ref{tab:model-params-a} and \ref{tab:model-params-b}. That means we carried out two experiments: the case of $\mathcal{R}_0>1$ and the case of $\mathcal{R}_0 < 1$.

In Fig.~\ref{fig:crime-alpha}, we observe that as we increase the value of the parameter $\alpha$, the number of criminal individuals also increases. On the other hand, if we decrease $\alpha$, the criminals will decline and approach zero. 

\begin{figure}[ht!]
    \centering
    \subfigure(a){\includegraphics[width=0.80\linewidth]{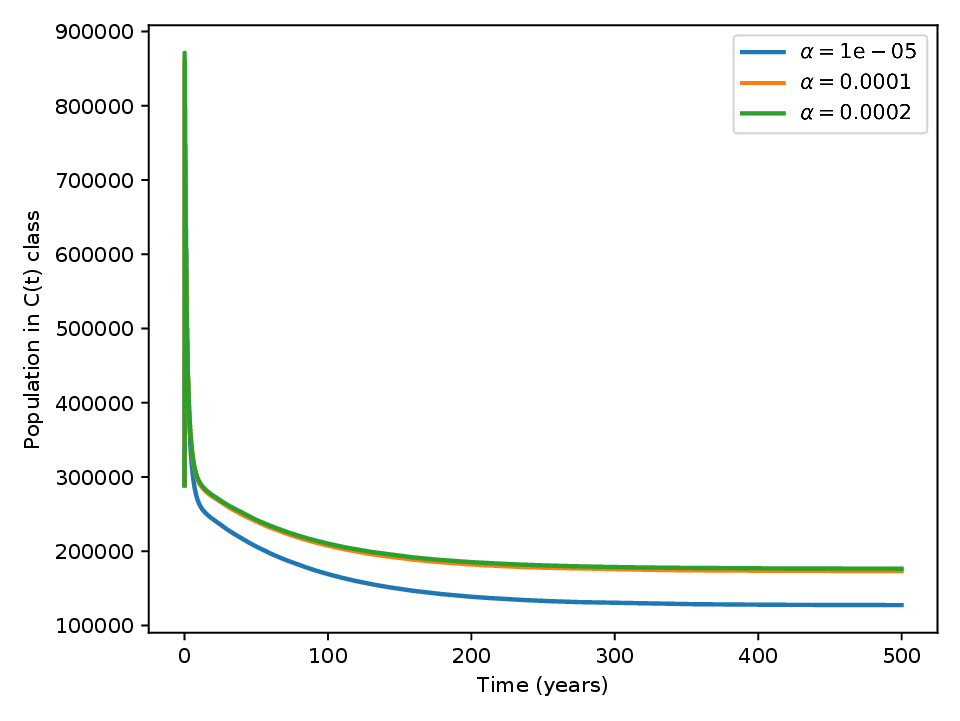}} 
    \subfigure(b){\includegraphics[width=0.80\linewidth]{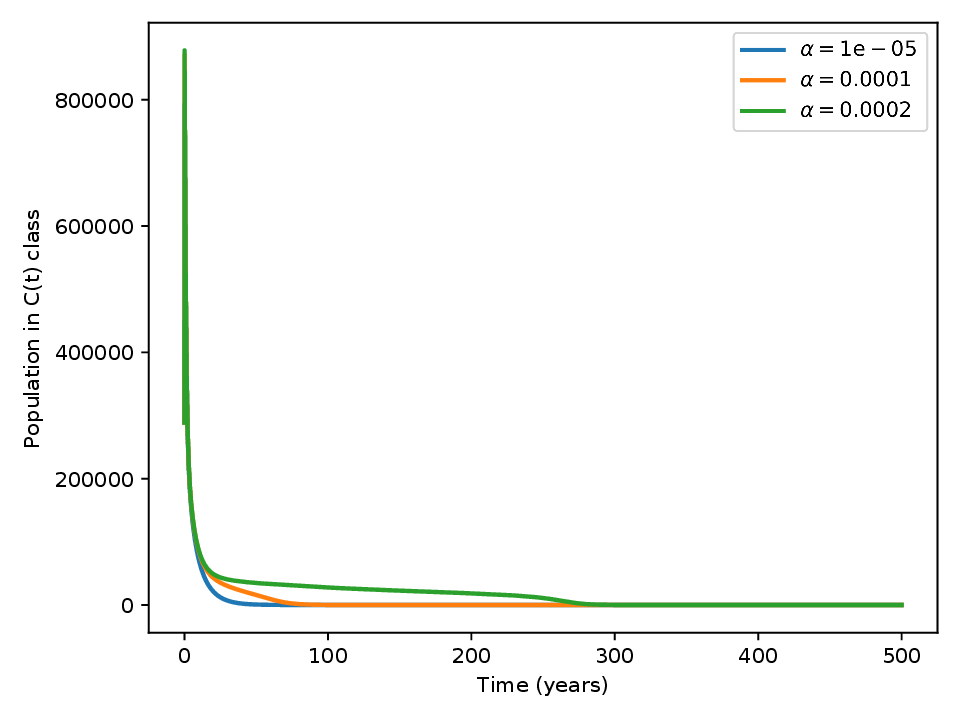}} 
    
    \caption{Impact of varying the imitation coefficient $\alpha$ on the dynamics of the criminality class $C(t)$. All parameters were fixed in both panels as provided in Table \ref{tab:model-params-a} and \ref{tab:model-params-a}, except that $\alpha$ was varied. In panel (a) $\mathcal{R}_0 \approx 1.5086 >0$ showing the stability of a unique endemic equilibrium. In panel (b) $\mathcal{R}_0 \approx 0.3762 <1$ and $\alpha^*\approx 0.0035 >\alpha$ for each value of $\alpha$ showing non-existence of endemic equilibrium solution. }
    \label{fig:crime-alpha}
\end{figure}

\section{Conclusion}
\label{sec:conclusion}

In this paper, we developed a mathematical model of crime dynamics in the presence of imitation, described by a system of nonlinear ordinary differential equations. Motivated by the epidemiological models of infectious diseases \cite{vanden}, our model is analyzed in terms of the basic reproduction number to suggest any controlling strategies for communities and policymakers to reduce the likelihood of an individual engaging in a criminal career. The crime-free equilibrium of the proposed model is locally asymptotically stable if the basic reproduction number is less than unity. A bifurcation analysis showed that the model exhibits a backward bifurcation. The backward bifurcation shows that it is insufficient to bring a basic reproduction number less than unity to clear crime. The basic reproduction number should be less than the critical threshold parameter $\mathcal{R}_{0}^c$ to clear any criminal act. If $\mathcal{R}_0$ is greater than unity, our proposed model has a unique endemic equilibrium. Our numerical simulations support our theoretical findings. In particular, the numerical results showed that the imitation coefficient significantly affects the spread of crimes. We also studied the sensitivity analysis of the basic reproduction number against the model parameters. The sensitivity analysis shows that the relapse rate highly influences $\mathcal{R}_0$. Therefore, as a controlling strategy to minimize any criminal activity, we should minimize the proportion of individuals leaving prisons and becoming criminals.

The model presented in this paper is not without limitations. The model assumes homogeneous mixing, which is practically impossible in communities with crimes. In reality, due to the differences in human behavior, the initiation of individuals into crime varies. We can also consider stochastic effects to model the unpredictability of human behavior. Hence, including stochasticity in human behavior can significantly improve this model. The recruitment of criminal members is assumed to be driven by imitation. But in practice, this is not the only means, as individuals can commit a criminal act by themselves due to forcing circumstances such as unemployment and poverty. When reliable crime-related data is available, we can validate the proposed model. We can extend the proposed model to include policies on managing convicts and the potential impact of correctional services. Despite these limitations, the model presents an exciting tool to track the dynamics of criminals and their convictions. 


\bibliographystyle{plain}
\bibliography{references}

\end{document}